\documentclass[12pt, reqno]{amsart}

\usepackage[text={400pt,660pt},centering]{geometry}

\usepackage{color}
\usepackage{esint,amssymb}
\usepackage{graphicx}
\usepackage{MnSymbol}
\usepackage{mathtools}
\usepackage[colorlinks=true, pdfstartview=FitV, linkcolor=blue, citecolor=blue, urlcolor=blue,pagebackref=false]{hyperref}
\usepackage{microtype}

\usepackage{bm}
\usepackage{scalerel} 
\usepackage{dsfont}
\usepackage{mathrsfs}
\usepackage{eufrak}
\usepackage[font={footnotesize}]{caption}

\definecolor{darkgreen}{rgb}{0,0.5,0}
\definecolor{darkblue}{rgb}{0,0,0.7}
\definecolor{darkred}{rgb}{0.9,0.1,0.1}

\newtheorem{theorem}{Theorem}
\newtheorem{proposition}{Proposition}
\newtheorem{lemma}[proposition]{Lemma}
\newtheorem{corollary}[proposition]{Corollary}
\newtheorem{claim}[proposition]{Claim}

\theoremstyle{remark}
\newtheorem{remark}[proposition]{Remark}

\theoremstyle{definition}
\newtheorem{definition}[proposition]{Definition}

\numberwithin{equation}{section}
\numberwithin{proposition}{section}

\newcommand{\Z}{\mathbb{Z}}
\newcommand{\N}{\mathbb{N}}
\newcommand{\R}{\mathbb{R}}
\newcommand{\E}{\mathbb{E}}

\renewcommand{\P}{\mathbb{P}}

\renewcommand{\subset}{\subseteq}

\DeclareMathOperator*{\argmax}{arg\,max}

\renewcommand{\bar}{\overline}
\renewcommand{\tilde}{\widetilde}

\newcommand{\indc}{\mathds{1}}
\newcommand{\1}{\mathds{1}}
\newcommand{\pbold}{\textbf{p}}

\begin{document}

\title[Random walk on the interchange process]{Random walk on a perturbation of the infinitely-fast mixing interchange process}

\begin{abstract}
We consider a random walk in dimension $d\geq 1$ in a dynamic random environment evolving as an interchange process with rate $\gamma>0$. We only assume that the annealed drift is non--zero. We prove that the empirical velocity of the walker $X_t/t$ eventually lies in an arbitrary small ball around the annealed drift if we choose $\gamma$ large enough. This statement is thus a perturbation of the case $\gamma =+\infty$ where the environment is refreshed between each step of the walker. We extend three-way part of the results of  \cite{HS15}, where the environment was given by the $1$--dimensional exclusion process: $(i)$ We deal with any dimension $d\geq 1$; $(ii)$ Each particle of the interchange process carries a transition vector chosen according to an arbitrary law $\mu$; $(iii)$ We show that $X_t/t$ is not only in the same direction of the annealed drift, but that it is also close to it.
\bigskip

\noindent  \emph{AMS  subject classification (2010 MSC)}: 
60K37, 
 82C22, 
  60Fxx, 
  82D30.


\thanks{ The present work was financially supported  by the European Union's Horizon 2020 research and innovation programme under the Marie Sklodowska-Curie Grant agreement No 656047.
 }

\author[M. Salvi]{Michele Salvi}
\address[M. Salvi]{Universit\'e Paris-Dauphine, PSL Research University, CNRS, UMR [7534], CEREMADE, 75016 Paris, France}
\email{salvi@ceremade.dauphine.fr}

\author[F. Simenhaus]{Fran\c cois Simenhaus}
\address[F. Simenhaus]{Universit\'e Paris-Dauphine, PSL Research University, CNRS, UMR [7534], CEREMADE, 75016 Paris, France}
\email{simenhaus@ceremade.dauphine.fr}

\end{abstract}

\subjclass[2010]{}

\maketitle

\section{Introduction, model and result}
%
We are interested in a Markovian walker on $\Z^d$ whose transitions are given by an underlying random environment. The environment itself evolves in time according to an interchange process: We initially put a single particle on each site of $\Z^d$ and suppose the particles carry a transition-probability vector chosen in an i.i.d.~way. We attach independent Poisson clocks with parameter $\gamma >0$ on each edge of $\Z^d$ and, when the clock attached to $\{x,y\}$ rings, the system is updated by exchanging the particles sitting in $x$ and $y$. 
We consider a time-discrete walker $(X_t)_{t\in\N}$ on top of this system: For each jump it chooses a direction according to the transition-probability vector carried by the particle on which it sits. We assume that the expected drift of the transition-probability vectors is non-zero. If the environment is completely refreshed at each step of $(X_t)$ (corresponding to $\gamma=+\infty$), the walker evolves in a homogenous environment and thus admits the expected drift as asymptotic velocity. We are concerned here with a perturbation of this regime. We actually prove that, choosing $\gamma$ large enough, the empirical velocity of the walker $X_t/t$ eventually lives in an arbitrary small ball around the expected drift. 

\smallskip

While the problem of the asymptotic velocity for random walks in random media has been widely studied in the case of static environments (see for example \cite{BCR} and references therein for the general i.i.d.~setting and \cite{BS13} for the reversible case), the case of  environments that evolve in time has gained a great attention from the mathematical community only in recent years. When the environment dynamics has good space-time mixing conditions, traps, i.e.~``unfavorable'' regions of the environment, tend to dissolve on a shorter time scale than the one that is interesting for the analysis of the walker, and therefore in many cases it has been possible to derive a law of large numbers (see, e.g., \cite{BZ06, DL09, AdHR11,AdSV13,RV13,dHdSS13} and references therein). It appears more complicated to study the asymptotic properties of a walker moving in an environment that mixes slowly, for example when the underlying dynamics is conservative. In this setting, some results have been lately obtained 
 for a walk on the  simple symmetric exclusion process with strong drift conditions in \cite{AdSV13}, for a walk moving on a collection of independent particles performing simple symmetric random walks in \cite{HdHdSST15}, for a walk moving on a super-critical contact process in \cite{dHdS14}, and for the perturbation of a homogeneous walk  on a class of dynamic environments satisfying the Poincar\'e inequality in \cite{ABF16}.
In \cite{HS15} the authors consider a one-dimensional dynamic random environment evolving as the simple exclusion process with jump parameter $\gamma$. A nearest-neighbor walker moves on this environment at integer times, jumping to the right (left) with probability $\alpha>0$ (resp.~$1-\alpha$) if it finds itself on a site occupied by a particle, and to the right (left) with probability $\beta>0$ (resp.~$1-\beta$) if it is on an empty site. The authors show that if in the case $\gamma=+\infty$ the walker has a strictly positive limiting speed, then, for $\gamma$ big enough, the walk will admit a strictly positive limiting speed, too. They also prove an analogous result (requiring only transience of the walk for $\gamma=0$) for $\gamma$ small enough.

\smallskip
In the present paper we do not prove a full law of large numbers: We postpone to a future work the construction of a suitable renewal structure (see e.g.~\cite{AdSV13,HdHdSST15,HS15}), which is the main tool to derive such a theorem. The present work can however be thought as a generalization of \cite[Theorem 1.1]{HS15} in three different directions:
\begin{itemize}
	\item[(i)] We pass from dimension $d=1$ to $d\geq 1$. We point out that most of the results for random walks in dynamic random environment have been achieved so far for dimension $d=1$, especially in the slowly-mixing dynamics framework;
	\item[(ii)] Instead of considering the simple exclusion process for the environment dynamics, we deal with the interchange process. In particular, while in \cite{HS15} only two possible transition-probability vectors are allowed, initially chosen with a Bernoulli product measure, in the present article we consider any vector coming from the $d$-dimensional simplex, sampled at time $0$ with an arbitrary product measure $\mu^{\otimes\Z^d}$;
	\item[(iii)] While in \cite{HS15} it is proven that, for $\gamma$ sufficiently large, $X_t/t$ is just eventually positive in the average drift direction, we prove here that, eventually, $X_t/t$ stays in fact close to the average drift.
\end{itemize}
 The present article finds a clear inspiration in \cite{HS15} and follows its general strategy. Nevertheless, in order to obtain the desired generalisations one needs to overcome several technical difficulties, especially due to high dimensionality. 
\subsection{Notation, model and main result.}
We call $\N:=\{0,1,...\}$ the set of natural numbers. We write $\mathcal I:=\{-1,\cdots,-d\}\cup\{1,\cdots,d\}$ and define $\mathcal N:=(e_i)_{i\in\mathcal I}\,$, where, for each $i=1,...,d$,  $e_i$ denotes the $i$-th element of the canonical basis of $\R^d$, while $e_{-i}:=-e_i$. We define $\mathcal S:=\{s\in[0,1]^{\mathcal I}:\,s_1+...+s_{d}+s_{-1}+...+s_{-d}=1\}$ to be the simplex of possible transition-probability vectors and we associate to each $s\in \mathcal S$ its drift 
\begin{equation*}
D(s):=\sum_{i\in\mathcal I}s_i\, e_i\,.
\end{equation*}
We denote by $\Omega:=(\mathcal S^{\Z^d})^{\R_+}$ the space of space-time environments. We think of an element $\omega$ as a time-indexed sequence of environments, where an environment is a collection of transition-probability vectors associated to each point of the lattice: $\omega_y(t)\in\mathcal S$ is the transition-probability vector in the site $y\in\Z^d$ at time $t\geq 0$, and $\omega^i_y(t)$ represents its $i$-th coordinate, $i\in \mathcal I$.

\smallskip

Given an environment $\omega \in \Omega$ we define the discrete-time random walk $(X_t)_{t\in\N}$ on $\Z^d$ with law $P_{\omega}$ as the process defined by the following: For all $k\in\N$, $y\in\Z^d$ and $i\in \mathcal I$ we have 
\begin{align}
&P_\omega(X_0=0)=1 \,,\nonumber\\
&P_\omega(X_{k+1}=y+e_i\,|\,X_k=y)=\omega^i_{y}(k)\,.\nonumber
\end{align}
Notice that, albeit $\omega$ has been defined for each positive real time $t$, the walk $(X_t)_{t\in\N}$ is influenced only by the environment at integer times. 

\smallskip
We now construct a probability measure on $\Omega$, corresponding to an interchange process where each particle carries a transition-probability vector chosen once and for all at time zero. We let $\mu$ be a probability measure on $\mathcal S$ and indicate with $\E_\mu[\cdot]$ the expectation w.r.t.~$\mu$. 
 Our unique assumption on $\mu$ is that
\begin{equation}
\label{eq:main assumption}
\E_{\mu}\left[D\right]\neq 0\,.
\end{equation}
In particular, we point out that we do not require any kind of ellipticity on the environment.
We consider the product law $\pi:=\mu^{\otimes\Z^d}$ on $\mathcal S^{\Z^d}$.
%
We finally construct a law $\P$ on $\Omega$ in the following way: We start from an $\eta\in\mathcal S^{\Z^d}$ chosen according to $\pi$, and we let it evolve according to the 
 generator 
$$
\mathcal Lf(\tilde\eta)
	=\frac\gamma 2\sum_{x,y\in\Z^d:\,\|x-y\|=1}\big(f(\tilde\eta\circ\sigma_{xy})-f(\tilde\eta)\big),
$$
where $\sigma_{xy}(\tilde\eta)$ indicates the environment obtained from $\tilde\eta\in\mathcal S^{\Z^d}$ by interchanging the particles in $x$ and $y$, i.e.~the values of $\tilde\eta_x$ and $\tilde\eta_y$. We indicate by $\|\cdot\|$ the usual $L^2$-norm. We indicate with $\P^\eta$ the law of the evolution of the environment started from a given $\eta\in\mathcal S^{\Z^d}$, $\E^\eta$ being the corresponding expectation. Notice that we will always use bold characters for indicating probabilities and expectations when we will deal with the sole environment process.
In our particle representation, we see that $\mathcal L$ is nothing but the generator of the interchange process: We can think to assign to each edge in $\Z^d$ an exponential clock of parameter $\gamma$, independent of all the other clocks. Whenever the clock of an edge $\{x,y\}$ rings, we proceed with exchanging the positions of the particles in $x$ and $y$.
We  define the law of the random walk in the dynamic environment starting from a given 
$\eta \in \mathcal S^{\Z^d}$ as
\begin{equation*}
P^{\eta}:= \P^{\eta} \times P_{\omega}\,,
\end{equation*}
and finally the annealed law, where we average over the initial configuration according to $\pi$, is
\begin{align*}
P:=\P\times P_{\omega}\,.
\end{align*}
Our goal is to prove the following theorem: 

\begin{theorem}
\label{thm:ballistic}
Under assumption \eqref{eq:main assumption}, for all $\varepsilon>0$, there exists $\gamma(\varepsilon)$ such that, for all $\gamma\geq \gamma(\varepsilon)$, $P$-a.s.~it holds that, for $t$ large enough,
$$
\frac{X_t}{t}\in B_\varepsilon\big(\E_\mu[D]\big)\,,
$$
where $B_\varepsilon(\E_\mu[D])$ is the $L^2$-ball of radius $\varepsilon$ around the averaged drift $\E_\mu[D]$.
\end{theorem}

\subsection{Outline of the paper.} 
In Section \ref{sec:auxiliary} we will prepare the technical  machinery needed for the proof of Theorem \ref{thm:ballistic}. The section is devoted only to the properties of the underlying interchange process: We say that a site $x\in \Z^d$ is $L$-good for a configuration $\eta\in \mathcal S^{\Z^d}$ if the empirical distribution of the environment in any ball of size larger than $L$ centred in $x$ is close to $\mu$, see Definition \ref{def:good} for a precise statement. In Proposition \ref{proposition1} we prove that if a site is $L$--good at time $0$, then it will be $J$--good ($J<L$) after a time $t$ satisfying $\gamma t>L^3$ with high probability. To this end, we will need a concentration inequality, Lemma \ref{lemma1}, and some estimates on the density of the particles carrying some type of transition-probability vector, Lemma \ref{lemma2}.
In Section \ref{sec:proof_of_theorem} we pass to the proof of Theorem \ref{thm:ballistic}.
 This will be achieved through a renormalization procedure. Proposition \ref{lemma4} deals with the initial step of the renormalization, giving the result up to any time $T$ choosing $\gamma$ large enough depending on $T$. Proposition \ref{prop:renorm} extends the result for any time $t>T$ by iteration. Finally, in Subsection \ref{proiezione} we finalize the proof.


\section{Auxiliary results on the environment dynamics}\label{sec:auxiliary}

In this section we will collect a series of results that concern only the environment dynamics.
We fix here and for the whole section an integer $N\in\N$. Note that all the statements in this section involve constants that may or may not depend on $N$, but we will not comment on that since it will not affect the results that we want to prove.
We consider a map $T:\,\mathcal S\to \mathcal T$, where we call $\mathcal T:=\{0,...,2^N-1\}^{{\mathcal I}}$ the space of \textit{types}, with $T:\,s\mapsto T(s)$ such that, for each coordinate $i\in\mathcal I$, 
$$
T_i(s):=\max\big\{j\in\{0,...,2^N-1\}:\,2^{-N}j\leq s_i\big\}\,.
$$
In particular, if $s_i\not=1$, then
 $ \frac{T_i(s)}{2^N} \leq s_i < \frac{T_i(s)+1}{2^N}$.
We simply say that an element $s\in\mathcal S$ is of type $k$ (where $k$ is a $2d$-dimensional integer multi-index) if $T(s)=k$ and call $p_k:=\mu(T(s)=k)$ the $\mu$-probability that a configuration $s$ is of type $k$. 
We point out that if $\mu$ gave positive weight only to a finite number of possible transition-probability vectors of $\mathcal S$, there would be no need to introduce types (notice for example that the simple exclusion process treated in \cite{HS15} corresponds to having only two possible transition-probability vectors). On the other hand, in order to deal with much more general $\mu$'s, the reduction to a finite number of types of particles is fundamental for our technique to work.

\smallskip

Given $x\in\Z^d$, $L\in\N$, $k\in\mathcal T$ and $\eta\in\mathcal S^{\Z^d}$, we let
$$
\langle\eta\rangle^k_{x,L}:=\frac{1}{|B_L(x)|}\sum_{y\in B_L(x)}\mathds{1}_{T(\eta_y)=k }
$$
be the empirical density of particles 
of type $k$ in a ball of radius $L$ around $x$. We consider $L^2$ balls: 
$$
B_L(x):=\{y\in\Z^d:\,\|x-y\|\leq L\},
$$
so that the cardinality of a ball is $ c_1 L^d\leq|B_L(x)|\leq  c_2 L^d$ for some constants $c_1,\,c_2>0$ depending on the dimension $d$.

\begin{remark}
For simplicity, throughout the article we will denote by $c,\,c_1,\,c_2,...$ some strictly positive constants the value of which might change from one expression to the other. These constants will not depend on the other variables involved (with the possible exception of $d$ and $N$), unless otherwise specified. 
\end{remark}

For simplicity, for  $L\in\N$, we abbreviate $\langle\eta\rangle^k_{L}:=\langle\eta\rangle^k_{0,L}$ and $B_L:=B_L(0)$.
We also fix a decreasing sequence $(\epsilon_L)_{L\in \N}$  of  real numbers in $(0,1)$ by 
\begin{equation}\label{epsilon}
\epsilon_L:=\frac{1}{1+\ln(L)}.
\end{equation}
The numbers $\epsilon_L$ are meant to control the
difference between the theoretical density $p_k$ and the empirical density of particles of  type $k$ in a box of size $L$:
\begin{definition}\label{def:good}
For $L\in\N$, we say that a site $x\in\Z^d$ is $L$-\textit{good} for a configuration $\eta\in\mathcal S^{\Z^d}$  if
\begin{align}
|\langle\eta\rangle_{x,L'}^k-p_k|\leq \epsilon_L\qquad \forall L'\geq L,\,\forall k\in\mathcal T.
\end{align}
In a configuration $\eta$, we call the set of $L$-good sites $G(\eta,L)$.
\end{definition}
\begin{remark}
Note that the definition of good sites depends on $N$.
\end{remark}
In words, $x\in G(\eta,L)$ if in every ball of radius $L'\geq L$ the density of particles of type $k$ is close to its theoretical mean, with an error of at most $\epsilon_L$, for every possible type $k$.
The typical deviation, under the equilibrium measure, of the empirical density in a $d$--dimensional ball of radius $L$ is of order $1/L^{d/2}$. Our conservative choice of $(\epsilon_L)$ in \eqref{epsilon} guarantees that a site has large probability to be good when $L$ is large.

\smallskip

The following proposition is one of the main tools that we will use for proving Theorem \ref{thm:ballistic}.

\begin{proposition}
\label{proposition1}
Suppose that $0\in G(\eta,L)$ for some $\eta\in\mathcal S^{\Z^d}$, $L\in\N$.
Then there exist $c_1,c_2>0$ such that the following is true: For all $J\in\N$,  $J<L$,  such that 
$L(\epsilon_J-\epsilon_L)$ is large enough, and for all $t>0$ such that $\gamma t>L^{3}$, it holds
\begin{align*}
\P^\eta\big(0\not\in G(\eta(t),J)\big)
	\leq c_1\frac{{\rm e}^{-c_2J^d(\epsilon_J-\epsilon_L)^2}}{(\epsilon_J-\epsilon_L)^2}.
\end{align*}
\end{proposition}
%
%
%
%
In order to prove Propositions \ref{proposition1} 
(which will be done in Subsection \ref{sec:proofs_of_the_propositions}) we will need the following 
two lemmas, that will be proven in Subsection \ref{prooflemma1} and Subsection \ref{prooflemma2}, respectively. Lemma \ref{lemma1} is a concentration-of-measure kind of statement, claiming that $\langle\eta(t)\rangle^k_{L}$ stays close to its mean. Lemma \ref{lemma2} controls the expectation $\E^\eta[\langle \eta(t)\rangle_{M}^k]$ in terms of the initial configuration $\eta$.

\begin{lemma}
\label{lemma1}
There exists a constant $c>0$ such that, given $\eta\in\mathcal S^{\Z^d}$, $L\in\N$, 
$k\in\mathcal T$ and $a\geq 0$, 
\begin{align*}
\P^\eta\big(\big|\langle\eta(t)\rangle^k_{L}-\E^\eta[\langle\eta(t)\rangle^k_{L}]\big|\geq a\big)\leq 2\,{\rm e}^{-c\,a^2L^d}.
\end{align*}
\end{lemma}

\begin{lemma}
\label{lemma2}
\label{le:lemma2}
For each $\alpha>0$, there exists a  constant $c>0$ such that, given $\eta\in\mathcal S^{\Z^d}$, $M,L\in\N$, $k\in\mathcal T$ and $t\geq 1$, the two following inequalities hold:
\begin{align}
\E^\eta\big[\langle \eta(t)\rangle_{M}^k\big]
	&\leq\Big(1+c\,\Big(\frac{L^{d+1}}{(\gamma t)^{(d+1)/{2}}}+\frac{1}{(\gamma t)^{1/2-\alpha }}\Big)\Big)
		\,\sup\Big\{   \langle \eta \rangle_{L'}^k:\,{L'\geq L} \Big\}\label{eqn:eta_upper_bound}\\
\E^\eta\big[\langle \eta(t)\rangle_{M}^k\big]	
	&\geq \Big(1-c\,\Big(\frac{L^{d+1}}{(\gamma t)^{(d+1)/{2}}}+\frac{1}{(\gamma t)^{1/2-\alpha }}\Big)\Big)
		\,\inf\Big\{   \langle \eta \rangle_{L'}^k:\,{L'\geq L} \Big\}.\label{eq:eta_lower_bound}
\end{align}
\end{lemma}
\begin{remark}
Lemma \ref{lemma1} provides a control of the average number of particles of a given type in a ball of size $M$ at time $t$ in terms of the initial particle configuration, for any $M\in\N$. It is not clear if the rates of decay in $t$ of the r.h.s.~of \eqref{eqn:eta_upper_bound} and of \eqref{eq:eta_lower_bound} are optimal:
While the term $(L/\sqrt{\gamma t})^{d+1}$ is  due to the diffusive dissipation of traps, the term $1/(\gamma t)^{1/2-\alpha}$ is present for technical reasons and it could perhaps be   improved.
%
 Nevertheless, not being interested here in sharp quantitative statements on $\gamma$, the  bounds \eqref{eqn:eta_upper_bound} and \eqref{eq:eta_lower_bound} will be more than sufficient for our purpose.
\end{remark}
%

%

\subsection{Proof of Lemma \ref{lemma1}}\label{prooflemma1}
We prove the two following inequalities, which imply the lemma:
\begin{align}
&\P^\eta\big(\langle\eta(t)\rangle^k_{L}-\E^\eta[\langle\eta(t)\rangle^k_{L}]\geq a\big)
	\leq {\rm e}^{-ca^2L^d}\label{prima},\\
&\P^\eta\big(\langle\eta(t)\rangle^k_{L}-\E^\eta[\langle\eta(t)\rangle^k_{L}]\leq -a\big)
	\leq {\rm e}^{-ca^2L^d}.\label{seconda}
\end{align}


Let us start with \eqref{prima}. We will replicate the argument used in \cite[Lemma 1]{HS15}, considering each $k$-particle of our system as a $1$-particle in \cite{HS15}. Even if there are only few minor modifications, we repeat here the whole proof in order to make the present article self-contained. For $0<\delta<|B_L|$, exponential Markov inequality gives
\begin{align}\label{eqn:markov}
\P^\eta\big(\langle\eta(t)\rangle^k_{L}-\E^\eta[\langle\eta(t)\rangle^k_{L}]\geq a\big)
	\leq {\rm e}^{-\delta a}\cdot{\rm e}^{-\delta\E^\eta[\langle\eta(t)\rangle^k_{L}]}\cdot
		\E^\eta\big[{\rm e}^{\delta \langle\eta(t)\rangle^k_{L}}\big].
\end{align}
Notice that we can write 
\begin{align}\label{eqn:meanofeta}
\E^\eta\big[\langle\eta(t)\rangle^k_{L}\big]
	&=\frac{1}{|B_L|}\sum_{y\in B_L}\E^\eta[\mathds{1}_{T(\eta_{y}(t))=k}] \nonumber\\
	 &=\frac{1}{|B_L|}\sum_{y\in B_L}\sum_{z\in\Z^d}p(t,y-z)\mathds{1}_{T(\eta_{z})=k} 
	= \sum_{z\in\Z^d}p_L(t,z)\mathds{1}_{T(\eta_{z})=k},
\end{align}
where $p(t,x)$ is the heat kernel associated to the Laplacian $\gamma\Delta$, i.e.~the probability that a simple random walk jumping at rate $\gamma$ and starting at the origin will be in site $x$ after time $t$, and 
\begin{align}\label{piel}
p_L(t,x):=\frac{1}{|B_L|}\sum_{y\in B_L}p(t,x+y)\,.
\end{align}
 
The third factor of \eqref{eqn:markov} can be handled through Ligget's inequalities (see \cite[Prop.~1.7, pag.~366]{L05}
): Let $\theta=(\theta(t))_{t\geq 0}$ be the collective motion of continuous-time independent simple random walks on $\Z^d$ with rate $\gamma$,  so that $\theta_x(t)$ indicates the number of particles on site $x$ at time $t$. We consider as starting configuration $\theta(0)$ the one such that $\theta_x(0)=\mathds{1}_{T(\eta_x)=k}$, i.e.~we let a simple random walk start in site $x$ if and only if the particle at site $x$ in $\eta$ is of type $k$. We arbitrarily label these random walks and call them $(X_j(t))_{t\geq 0}$, with $j\in\N$. 
We indicate the empirical average of the number of particles at time $t$ in a box of size $L$ around the origin as $\langle\theta(t)\rangle_{L}$. Set $\pbold_L(t,x):=|B_L|\cdot p_L(t,x)$.
Using Liggett's inequality for the first line, we can now bound
\begin{align}\label{eqn:thirdfactor}
\E^\eta\big[{\rm e}^{\delta \langle\eta(t)\rangle^k_{L}}\big]
	&\leq \E^\eta\big[{\rm e}^{\delta \langle\theta(t)\rangle_{L}}\big]
	=\prod_{j\in\N}\E^\eta\big[{\rm e}^{\frac\delta{|B_L|} \mathds{1}_{X_j(t)\in B_L}}\big]				\nonumber\\
	&=\prod_{j\in\N}\Big({\rm e}^{\frac\delta{|B_L|}} \pbold_L(t,X_j(0))
		+\big(1- \pbold_L(t,X_j(0))\big)\Big)\nonumber\\
	&=\prod_{z\in\Z^d}\Big({\rm e}^{\frac\delta{|B_L|}} \pbold_L(t,z)
		+(1- \pbold_L(t,z))\Big)^{\indc_{T(\eta_z)=k}}\nonumber\\
	&\leq \exp\Big\{  ({\rm e}^{\frac\delta{|B_L|}}-1)\sum_{z\in\Z^d}\pbold_L(t,z)					\indc_{T(\eta_z)=k}\Big\}\nonumber\\
	&=\exp\Big\{({\rm e}^{\frac\delta{|B_L|}}-1)\,|B_L|\,\E^\eta[\langle\eta(t)\rangle^k_{L}]\Big\}\,,
\end{align}
where the last equality follows from \eqref{eqn:meanofeta}.
Since $\delta/|B_L|\leq 1$, we can expand ${\rm e}^{\frac\delta{|B_L|}}-1$ to the first order in \eqref{eqn:thirdfactor} and find a constant $C<\infty$ such that
$$
 \E^\eta\big[{\rm e}^{\delta \langle\eta(t)\rangle^k_{L}}\big]
	\leq {\rm e}^{\delta \E^\eta[\langle\eta(t)\rangle^k_{L}]+C\frac{\delta^2}{|B_L|} \E^\eta[\langle\eta(t)\rangle^k_{L}]}\,.
$$
Plugging this back into \eqref{eqn:markov}, we get
\begin{align*}
\P^\eta\big(\langle\eta(t)\rangle^k_{L}-\E^\eta[\langle\eta(t)\rangle^k_{L}]\geq a\big)
	\leq  {\rm e}^{-a\delta+C\frac{\delta^2}{|B_L|} \E^\eta[\langle\eta(t)\rangle^k_{L}]}\,.
\end{align*}
Minimizing this expression in $\delta\in[0,|B_L|]$ (see \cite[Proof of Lemma 2.3]{HS15} for more details) and remembering that $ c_1 L^d\leq|B_L|\leq  c_2 L^d$, we obtain the claim.

\smallskip

We are left to show \eqref{seconda}.
To this end, we just consider all the particles of type $j\not = k$ as being of the same type. We call $\langle\eta(t)\rangle^{\hat k}_{L}:=1-\langle\eta(t)\rangle^k_{L}$ the empirical average of particles of type different form $k$ in $B_L$ and notice that
\begin{align*}
\P^\eta\big(\langle\eta(t)\rangle^k_{L}-\E^\eta[\langle\eta(t)\rangle^k_{L}]\leq -a\big)
	=\P^\eta\big(\langle\eta(t)\rangle^{\hat k}_{L}-\E^\eta[\langle\eta(t)\rangle^{\hat k}_{L}]\geq a\big).
\end{align*}
Then we just have to replay exactly the same game as above.

\medskip


\subsection{Proof of Lemma \ref{lemma2}}\label{prooflemma2}

Recall the definition \eqref{piel}. We divide $\Z^d$ in disjoint crowns centered in $0$: $\Z^d=\bigcup_{n\in\N}C_n$ with $C_n:=\{z\in\Z^d:\, n-1< \|z\|\leq n\}$.
For $y\in\Z^d$ we define the crown of $y$ as $C(y):=C_{n(y)}$ where $n(y)$ is the unique integer such that $y\in C_{n(y)}$.
For each $n\in\N$ we choose representatives $\hat x_n,\,\check x_n\in C_n$ such that 
\begin{align*}
p^+_M(t,n):=p_M(t,\hat x_n)=\max_{z\in C_n}\{p_M(t,z)\}\,,
\quad\mbox{}
p^-_M(t,n):=p_M(t,\check x_n)=\min_{z\in C_n}\{p_M(t,z)\}.
\end{align*}
We also define, for $y\in\Z^d$, 
$p^+_M(t,y):=\max_{z\in C(y)}\{p_M(t,z)\}$ and
$p^-_M(t,y):=\min_{z\in C(y)}\{p_M(t,z)\}.
$

\begin{claim}\label{claim:heatkernel_decrease}
Take $y\in\Z^d$ such that $\langle y,e_i\rangle>0$ for some $i\in \mathcal I$. Then
\begin{align*}
p_M(t,y)\geq p_M(t,y+e_i).
\end{align*}
\end{claim}
\begin{corollary}[Corollary of the Claim]\label{cor:crowns}
\begin{align*}
p^+_M(t,y)\geq p^+_M(t,y')\quad \mbox{and} \quad p^-_M(t,y)\geq p^-_M(t,y')
\qquad \forall y, y'\in\Z^d:\,\|y'\|\geq \|y\|.
\end{align*}
\end{corollary}

\smallskip

\noindent We postpone the easy proofs of Claim \ref{claim:heatkernel_decrease} and Corollary \ref{cor:crowns} to, respectively, Subsection \ref{appendix1} and Subsection \ref{appendix2} in the Appendix.

\smallskip

With Corollary \ref{cor:crowns} at hand, we would like to proceed in a similar fashion as in \cite{HS15}, but the  high dimensionality represents an obstacle to this end as, due to the geometry of the lattice, the particles starting in a given crown do not have exactly the same probability to reach the ball of radius $M$ at a given time. For each positive function $u:\Z^d\to\R$, we can write
\begin{align}\label{eq:sonice}
p_M(t,y)u(y)\leq p^+_M(t,y)u(y)=\sum_{n\in\N}\big(p^+_M(t,n)-p^+_M(t,{n+1})\big)u(y)\mathds{1}_{\|y\|\leq n}.
\end{align}
If we sum over all $y\in\Z^d$, take $u(y)= \mathds{1}_{T(\eta_y)=k}$ and 
use \eqref{eqn:meanofeta}, 
we obtain
\begin{align}\label{eq:boundetabar}
\E^\eta[\langle \eta(t)\rangle_{M}^k]
	&\leq \sum_{n\in\N}\big(p^+_M(t,n)-p^+_M(t,{n+1})\big)
		\sum_{y\in\Z^d}\mathds{1}_{T(\eta_y)=k}\mathds{1}_{\|y\|\leq n}\nonumber\\
	&= \sum_{n\in\N}\big(p^+_M(t,n)-p^+_M(t,{n+1})\big)|B_n| \langle \eta\rangle_{n}^k.\end{align}

We abbreviate $R:= \sup\{   \langle \eta \rangle_{L'}^k:\,{L'\geq L} \}$ and split the final sum in \eqref{eq:boundetabar} for $n<L$ and $n\geq L$.
For $n<L$ we have that $|B_n| \langle \eta\rangle_{n}^k\leq |B_L|R$, since the r.h.s.~is bigger than the total number of particles of type $k$ at time $0$ in the ball of radius $L$ and this clearly dominates the number of particles in any smaller ball. Hence,
\begin{align}\label{eq:n_smaller_L}
\sum_{n< L}\big(p^+_M(t,n)-p^+_M(t,{n+1})\big)|B_n| \langle \eta\rangle_{n}^k
	&\leq \big(p^+_M(t,0)-p^+_M(t,L)\big)|B_L|R 
	\leq \frac{cL^{d+1}}{(\gamma t)^{\frac{d+1}2}}R.
\end{align}
To see why the last inequality holds, we expand, for any $z\in\Z^d$ and $t>0$, in Fourier variables $p(t,z)={\rm e}^{\gamma t\Delta}\delta_0(z)$. We call $\alpha(\xi):=\sum_{j=1}^d2(\cos(2\pi\xi\cdot e_j)-1)$ and bound, for each $z\in\Z^d$, $e\in\mathcal N$, $t>0$,
\begin{align}\label{eq:fourier}
\big|p(t,z)-p(t,z+e)\big |
	&=\Big|\int_{\R^d}\frac{{\rm e}^{\gamma t \alpha(\xi)}}{(2\pi)^d}{\rm e}^{i\xi\cdot z}
		(1-{\rm e}^{i\xi\cdot e})\,{{\rm d}\xi}\Big|
	\leq c \,\int_{\R^d}|\xi|{\rm e}^{-c_1\gamma t|\xi|^2}{\rm d}\xi 
	\leq \frac{c_2}{(\gamma t)^{\frac{d+1}{2}}}\,.
	\end{align}
 Then we are done, since \eqref{eq:fourier} is uniform in $z$ and $e$ and since we can find a path of neighbors leading from any point $y\in B_M$ to the point $\hat x_L+y$ in $\sqrt dL$ steps, which gives $p^+_M(t,0)-p^+_M(t,L)\leq cL/(\gamma t)^{(d+1)/2}.$


\smallskip

We want to show now that
\begin{align}\label{colorado}
\sum_{n\geq L}\big(p^+_M(t,n)-p^+_M(t,{n+1})\big)|B_n| \langle \eta\rangle_{n}^k
	&\leq \Big(1+c\,\Big(\frac{L^{d}}{(\gamma t)^{\frac{d+1}{2}}}+\frac{1}{(\gamma t)^{\frac 12-\alpha}}\Big)\Big)\,R\,.
\end{align}
This together with \eqref{eq:n_smaller_L} and \eqref{eq:boundetabar} will imply the first half of the lemma, inequality \eqref{eqn:eta_upper_bound}.
For $n\geq L$  we have that $\langle \eta\rangle_{n}^k\leq R$ by definition and we are left to bound $\sum_{n\geq L}(p^+_M(t,n)-p^+_M(t,{n+1}))|B_n|$.
By summing by parts we have
\begin{align}\label{eq:valid_for_Qq}
\sum_{n\geq L}\big(&p^+_M(t,n)-p^+_M(t,{n+1})\big)|B_n|
	=\sum_{n\geq L} p^+_M(t,n)|C_n|+p^+_M(t,L)|B_{L-1}|\nonumber\\
	&=\sum_{y\not\in B_{L-1}}
	p_M(t,y)
		+\sum_{n\geq L} \sum_{y\in C_n}\big(p^+_M(t,n)-p_M(t,y)\big)
		+p^+_M(t,L)|B_{L-1}|\,.
\end{align}
Since by Corollary \ref{cor:crowns} we have, for each $y\in C_m$ with $m<L$, that $p_M(t,y)\geq p_M^-(t,m)\geq p^-_M(t,L)$, we can bound 
$
p^+_M(t,L)|B_{L-1}|\leq \big(p^+_M(t,L)-p^-_M(t,L)\big)|B_{L-1}|+\sum_{y\in B_{L-1}} p_M(t,y)\,,
$
which inserted into \eqref{eq:valid_for_Qq} gives
\begin{align}\label{napoleon}
\sum_{n\geq L}\big(&p^+_M(t,n)-p^+_M(t,{n+1})\big)|B_n|
	\leq 1+\sum_{n\geq L}E_M(t,n)|C_n|+E_M(t,L)|B_{L-1}|\,,
\end{align}
where we have used the fact that $\sum_{y\in\Z^d}p_M(t,y)=1$ and where we have called 
$$
E_M(t,n):=p_M(t,\hat x_n)-p_M(t,\check x_n)
$$ 
the maximum error we commit in the $n$-th crown. 
%
%
%
%
%
%
\begin{claim}\label{climatizzo}
For each $\varepsilon>0$ there exist constants $c_1$ and $c_2$ such that
\begin{align}\label{dumbo}
\sum_{n\geq L}E_M(t,n)|C_n|\leq c_1\,\frac{1}{(\gamma t)^{1/2-\varepsilon d}}
\end{align}
and 
\begin{align}\label{dumbo2}
E_M(t,L)|B_{L-1}|\leq c_2\,\frac {L^{d}}{(\gamma t)^{(d+1)/2}}\,.
\end{align}
\end{claim}

The proof of Claim \ref{climatizzo} is rather technical and we postpone it to  Subsection \ref{appendix3} in the Appendix. We show now how to conclude the proof of Lemma \ref{lemma2}. Plugging \eqref{dumbo} and \eqref{dumbo2} into \eqref{napoleon} we obtain \eqref{colorado} with $\alpha=\varepsilon d$. This implies \eqref{eqn:eta_upper_bound} as mentioned before.

\smallskip

The lower bound \eqref{eq:eta_lower_bound} can be proven in an almost identical way as the upper bound: 
As in \eqref{eq:sonice}, we first note that for any positive $u$ one has $p_M(t,y)u(y)\geq\sum_{n\in\N}(p^-_M(t,n)-p^-_M(t,{n+1}))u(y)\mathds{1}_{n\geq \|y\|}$ and obtain a lower bound with the same flavor of \eqref{eq:boundetabar}. Then one splits once again the sum in $n$: This time the case $n< L$ can be lower-bounded directly with $0$. For the case $n\geq L$ we follow the same calculation as in \eqref{eq:valid_for_Qq} and produce a similar lower bound that this time reads:
$$
\sum_{n\geq L}(p^-_M(t,n)-p^-_M(t,{n+1}))|B_n|
	\geq 1-\sum_{n\geq L}E_M(t,n)|C_n|- E_M(t,L)|B_{L-1}|\,.
$$
Then we can use again the results of Claim \ref{climatizzo} and proceed exactly as for the upper bound \eqref{eqn:eta_upper_bound}, concluding the proof of the lemma.

\subsection{Proof of Proposition \ref{proposition1}}\label{sec:proofs_of_the_propositions}

We extend the argument in \cite{HS15} to the case of $N$ different kinds of particles. By union bound
\begin{align}\label{eq:originedelmondo}
\P^\eta\big(0\not\in G(\eta(t),J)\big)
	\leq \sum_{k\in \mathcal T}\sum_{J'\geq J} \P^\eta\big(|\langle\eta(t)\rangle_{J'}^k-p_k|>\epsilon_{J}\big)\,.
\end{align}
Notice that if $p_k=0$ for some $k\in\mathcal T$, the relative probability in the r.h.s.~is just zero. For all the other $k$'s we want to bound
\begin{align}\label{eq:absolute_positive}
&\P^\eta\big(|\langle\eta(t)\rangle_{J'}^k-p_k|>\epsilon_{J}\big)= \P^\eta\big(\langle\eta(t)\rangle_{J'}^k>p_k+\epsilon_{J}\big)
		+ \P^\eta\big(\langle\eta(t)\rangle_{J'}^k<p_k-\epsilon_{J}\big)\,.
\end{align}
For the first summand we rewrite
\begin{align}\label{complimentoni}
\P^\eta\big(\langle\eta(t)\rangle_{J'}^k>p_k+\epsilon_{J}\big)
	&= \P^\eta\big(\langle\eta(t)\rangle_{J'}^k-\E^\eta[\langle\eta(t)\rangle^k_{J'}]> p_k+\epsilon_{J}-\E^\eta[\langle\eta(t)\rangle^k_{J'}]\big)\,.
\end{align}
Since by hypothesis $\gamma t>L^{3}$ and $0\in G(\eta,L)$, Lemma \ref{lemma2} with $\alpha$ small enough guarantees that 
$
\E^\eta[\langle\eta(t)\rangle^k_{J'}]
	\leq (1+c/L)\sup\{\langle\eta\rangle_{L'}^k, L'\geq L\}
	\leq(1+c/L)(p_k+\epsilon_L)
$
so that 
$
p_k+\epsilon_J-\E^\eta[\langle\eta(t)\rangle^k_{J'}]
	\geq \epsilon_J-\epsilon_L-{c(p_k+\epsilon_L)}/{L}
	\geq (\epsilon_J-\epsilon_L)/{2},
$
where for the last inequality we have used the 
hypothesis that $L(\epsilon_J-\epsilon_L)$ is large enough. This applied to \eqref{complimentoni} gives
\begin{align}\label{stessacosa1}
\P^\eta\big(\langle\eta(t)\rangle_{J'}^k>p_k+\epsilon_{J}\big)
	&\leq \P^\eta\left(\big|\langle\eta(t)\rangle_{J'}^k-\E^\eta[\langle\eta(t)\rangle^k_{J'}]\big|\geq (\epsilon_J-\epsilon_L)/{2}\right)
	\leq  2\,{\rm e}^{-c (\epsilon_J-\epsilon_L)^2J'^d },
\end{align}
where the last inequality follows from Lemma \ref{lemma1}.

\smallskip

Conversely, we bound the second summand of \eqref{eq:absolute_positive} by noticing that, by Lemma  \ref{lemma2},
$
\E^\eta[\langle\eta(t)\rangle^k_{J'}]
	\geq (1-c/L)\inf\{\langle\eta\rangle_{L'}^k, L'\geq L\}
	\geq(1-c/L)(p_k-\epsilon_L),
$
so that
$
p_k-\epsilon_J-\E^\eta[\langle\eta(t)\rangle^k_{0,J'}]
	\leq -(\epsilon_J-\epsilon_L)/2<0
$.
Hence, as before,
\begin{align}\label{eq:stessacosa2}
\P^\eta\big(\langle\eta(t)\rangle_{J'}^k<p_k-\epsilon_{J}\big)
	&\leq \P^\eta\big(\big|\langle\eta(t)\rangle_{J'}^k-\E^\eta[\langle\eta(t)\rangle^k_{0,J'}]\big| 	>  (\epsilon_J-\epsilon_L)/{2}\big)
		\leq  2\,{\rm e}^{-c  (\epsilon_J-\epsilon_L)^2J'^d}.
\end{align}
Plugging \eqref{stessacosa1} and \eqref{eq:stessacosa2} into \eqref{eq:absolute_positive} and this back into \eqref{eq:originedelmondo} we finally obtain
\begin{align*}
\P^\eta\big(0\not\in G(\eta(t),J)\big)
	&\leq \sum_{J'\geq J} \sum_{k\in \mathcal T:\,p_k>0}
		4\,{\rm e}^{-c J'^d (\epsilon_J-\epsilon_L)^2}\nonumber\\
	&\leq c_1\,\sum_{k\in \mathcal T:\,p_k>0}\frac{{\rm e}^{-c J^d (\epsilon_J-\epsilon_L)^2}}{(\epsilon_J-\epsilon_L)^2}
	\leq c_1\,(2^{N})^{2d} \frac{{\rm e}^{-c J^d(\epsilon_J-\epsilon_L)^2}}{(\epsilon_J-\epsilon_L)^2}.\qedhere
\end{align*}
\bigskip

\section{Proof of Theorem \ref{thm:ballistic}}\label{sec:proof_of_theorem}

Given a vector $\vec v\in\R^d$, we denote, for $t\geq 0$,
\begin{align*}
X_t^{\vec v}:=\langle X_t,\vec v \rangle\,,
\end{align*}
where $\langle\cdot,\cdot\rangle$ denotes the usual inner product.
In order to prove Theorem \ref{thm:ballistic}, we will actually prove the following proposition, which is clearly equivalent to Theorem \ref{thm:ballistic}: 
\begin{proposition}[Equivalent of Theorem \ref{thm:ballistic}]
\label{prop:projection}
Assume \eqref{eq:main assumption} and consider $\vec v \in \R^d$ with norm $1$ such that $v:=\langle\E_{\mu}[D],\vec v\rangle>0$. For all $\varepsilon>0$, there exists $\gamma(\varepsilon)$ such that, for all $\gamma\geq \gamma(\varepsilon)$, $P$-a.s.~it holds that, for $t$ large enough,
\begin{align*}
\frac{X_t^{\vec v}}{t}\in (v-\varepsilon,v+\varepsilon)\,.
\end{align*}

\end{proposition}
\noindent The proof of Proposition \ref{prop:projection} is postponed to Subsection \ref{proiezione}.

\smallskip 

Without being too rigorous, we call \textit{trap} a region where an anomalous density of some type of particles occurs.  If all sites of a given region are $L$-good for some $L>0$, we say that there are no traps larger than $L$ in that region. 
We define now a sequence of positive numbers $(\phi_L)_{L\in\N}$. One can think of $\phi_L$ as the size of ``acceptable'' traps in a ball of radius $L$.  
We are not interested in a sharp estimate of the typical size of the maximal trap (that should be of order $\log L$). Instead, in order to simplify our computation, we choose rather to overestimate  this typical size and define, for $L \in\N$, 
\begin{align}\label{cento}
\phi_L:= L^{1/100}.
\end{align}

\subsection{Outline of the proof}
The proof relies on a renormalisation procedure adapted from \cite{HS15}.
We fix throughout the rest of the section a vector $\vec v$ satisfying the assumption of Proposition \ref{prop:projection}. 
 Heuristically, the hypothesis we want to iterate is the following:
  ``If the box of size $t$ at time $0$ around the walker has no trap larger than $\phi_t$, then, at time $t$, the walker is, with large probability, at a distance $(v+o(1))\,t$ from the origin   in the direction $\vec v$".

The first step of the procedure (see Proposition \ref{lemma4}) relies on a homogenization obtained by choosing a large enough finite time $T$ and  some  $\gamma$  large enough chosen accordingly. To prove the iteration itself (which consists in going from scale $t$ to scale $t^2$) we divide the space-time box of size $t^2$ into boxes of size $t$ and use Proposition \ref{proposition1}  to control the size of the traps in each of these sub-boxes. This allows us to apply the iteration hypothesis at scale $t$. 

\subsection{Initialization of the renormalization}
\begin{proposition}
\label{lemma4} 
For all $\varepsilon>0$, there exists $N$ large enough,  $ T\in\N$ large enough and $\gamma>0$ large enough (e.g., $\gamma \geq \phi_T^{3}$) such that the following holds: Given $\eta\in\mathcal S^{\Z^d}$ satisfying
$$
 y\in G(\eta,\phi_T)\qquad \forall y\in B_T,
$$
it holds that 
\begin{align*}
P^\eta\big(|X_T^{\vec v}-vT|\geq\varepsilon T\big)
	\leq {\rm e}^{-\phi_T^{1/4}}.
\end{align*}
\end{proposition}

\begin{proof}
We will first prove that 
\begin{align}\label{eq:half_lemma}
P^\eta(X_T^{\vec v}\leq (v-\varepsilon)T)
	\leq \tfrac 12\,{\rm e}^{-\phi_T^{1/4}}.
\end{align}
Since we have no information on the particle at the origin at time $0$, we just bound the first step with the worst case:
$$
P^\eta\big(X_T^{\vec v}\leq (v-\varepsilon)T\big)\leq P^\eta\big(X_T^{\vec v}-X_1^{\vec v}\leq (v-\varepsilon)T+1\big).
$$

For $j\in\mathcal I$ we let 
$v_j:=\langle \vec v, e_j\rangle\in[-1,1]$. For simplicity we suppose that $v_j\not=v_i$ for all $i\not =j$ (the proof can be easily adapted to the other cases).
We also define $v_{\max}:=\max_{j\in\mathcal I}\{v_j\}$, $v_{\min}:=-v_{\max}$ and their relative indices $j_{\max}:=\argmax v_{\max}$, $j_{\min}:=-j_{\max}$. 
By eventually enlarging the original space, we consider a sequence $(Y_k)_{k\in\N}$ of i.i.d.~random variables under $P^\eta$ with values in $\{v_{-d} ,...,v_{d}\}$. The $Y_k$'s are independent of $\eta$, of the interchange process and of the random walker, and they satisfy
\begin{align*}
P^\eta(Y_1=v_j)=
\begin{cases}
\E_\mu [s_j]-\delta \qquad & \mbox{for }j\not= j_{\min} \\
\E_\mu [s_{j_{\min}}]+(2d-1)\delta\qquad & \mbox{for }j=j_{\min},
\end{cases}
\end{align*}
where $\delta>0$ is chosen small enough and $N\in\N$ is chosen large enough such that $2^{-(N-1)}d<\delta<\min\{\E_\mu[s_i]>0\}$ and $2dv_{\max}\delta<\varepsilon/2$. 
Notice that:

\smallskip 

(i) $Y_1$ is well defined, or can be made well defined with little effort. In fact, if $\E_\mu[s_j]>0$ for all $j\in\mathcal I$, $Y_1$ is well defined because of the choice of $\delta$ that we made. If instead there is some $j\not=j_{\min}$ such that $\E_\mu[s_j]=0$, then we can easily modify the law of  $Y_1$ so that we define $P^\eta(Y_1=v_j)=0$ and adapt the definition of $P^\eta(Y_1=v_{\min})$ consequently;

\smallskip

(ii) $E^\eta[Y_1]>v-\varepsilon/2$. In fact, it is straightforward to check that
\begin{align*}
E^\eta[Y_1]
	&=v+\delta\Big((2d-1)v_{\min}-\sum_{j\not = j_{\min}}v_j\Big)=v-2d v_{\max}\delta>v-\varepsilon/2,
\end{align*}
where we have used the symmetry of the $v_\cdot$'s for the second equality, which implies $\sum_{j\not = j_{\min}}v_j=v_{\rm max}$ and $v_{\rm min}=-v_{\rm max}$. 

\smallskip

In analogy with \cite{HS15}, for $m\geq 1$ we define the events 
\begin{align*}
\mathcal E_{m} :=\big\{y\in G(\eta(m),J),\,\forall y\in B_{m+1}\big\}\qquad \mbox{with } J:=\lfloor\phi_T^{1/2}\rfloor.
\end{align*}
and $\mathcal E_{0} :=\big\{y\in G(\eta(0),\phi_T),\,\forall y\in B_1\big\}$.
Notice that, by hypothesis, $P^\eta(\mathcal E_0)=1$. Loosely speaking, the event $\mathcal E^c_m$ indicates that at step $m$ the random walk sees around itself an unfavorable environment. Our aim is to show that 
\begin{align}\label{eq:bye_badevents}
P^\eta(X_T^{\vec v}-X_1^{\vec v}\leq (v-\varepsilon)T+1) 
	&\leq P^\eta\Big(\sum_{m=1}^{T-1}Y_m\leq (v-\varepsilon)T+1\Big) +\sum_{m=1}^{T-1}\P^\eta(\mathcal E^c_{m-1}).
\end{align}
Before showing why this is true, we explain how this implies the first half of the lemma, i.e.~equation \eqref{eq:half_lemma}. The first term of the right-hand-side of \eqref{eq:bye_badevents} can be bounded with classical large deviations for sequences of i.i.d.~random variables: Keeping in mind that $E^\eta[Y_1]\in(v-\varepsilon/2,v)$, we have
\begin{align}\label{eq:ld_iid}
P^\eta\Big(\sum_{m=1}^{T-1}Y_m\leq (v-\varepsilon)T+1\Big) \leq {\rm e}^{-cT}
	\leq 	{\rm e}^{-\phi_T^{1/3}}\,.
\end{align}
The second term in the right hand side of \eqref{eq:bye_badevents} can be bounded thanks to Proposition \ref{proposition1} and by taking $T$ and, consequently, $\gamma$ large enough as
\begin{align}\label{eq:unfavorable}
\sum_{m=1}^{T-1}\P^\eta(\mathcal E^c_{m-1})
	&\leq \sum_{m=1}^{T-1}\sum_{y\in B_m} \P^\eta\big(y\not\in G(\eta(m-1),J)\big)
	\leq  c\,\sum_{m=1}^{T-1}\sum_{y\in B_m}
	\frac{{\rm e}^{-cJ^d(\varepsilon_J-\varepsilon_{\phi_T})^2}}{(\varepsilon_J-\varepsilon_{\phi_T})^2}\nonumber\\
	&\leq c \,T^{d+1} \frac{{\rm e}^{-cJ^d(\varepsilon_J-\varepsilon_{\phi_T})^2}}{(\varepsilon_J-\varepsilon_{\phi_T})^2}
	\leq {\rm e}^{-\phi_T^{{d}/3}}.
\end{align}
Putting back \eqref{eq:ld_iid} and \eqref{eq:unfavorable} into \eqref{eq:bye_badevents} we obtain \eqref{eq:half_lemma}.

\smallskip

Let us go back to the proof of \eqref{eq:bye_badevents}. We show that, for any $m\geq 1$ and any $a\in\R$,
\begin{align}\label{eq:m_goodorbad}
P^\eta(X_{m+1}^{\vec v}-X_1^{\vec v}\leq a)\leq P^\eta(X_{(m-1)+1}^{\vec v}+Y_m-X_1^{\vec v}\leq a)+P^\eta(\mathcal E^c_{m-1})\,.
\end{align} 
Iterating this formula and using Fubini gives \eqref{eq:bye_badevents}.
We bound
\begin{align}
\label{eq:but}
P^\eta&(X_{m+1}^{\vec v}-X_1^{\vec v}\leq a)\nonumber\\
	&\leq E^\eta\big[P^\eta(X_{m+1}^{\vec v}-X_m^{\vec v}+X_m^{\vec v}-X_1^{\vec v}\leq a\,|\,\eta(m-1),X_1,X_m)\cdot\1_{\mathcal E_{m-1}}\Big]+P^\eta(\mathcal E_{m-1}^c)\,.
\end{align} 
For any   $b\in \R$ we call $\mathcal J(b):=\{j\in\mathcal I:\,v_j\leq b\}$ and compute
\begin{align}\label{eq:new}
P^\eta(&X_{m+1}^{\vec v}-X_m^{\vec v}\leq b\,|\,\eta(m-1),X_1,X_m) \nonumber\\
	&= \sum_{j\in \mathcal \mathcal J(b)}\sum_{k\in\mathcal{T}}
		P^\eta\big(X_{m+1}^{\vec v}-X_m^{\vec v}=v_j\,|\,\eta(m-1),X_1,X_m,T(\eta_{X_m}(m))=k\big)\nonumber\\
	&\qquad \qquad\qquad \qquad\qquad\qquad\qquad P^\eta(T(\eta_{X_m}(m))=k\,|\,\eta(m-1),X_1,X_m)\nonumber\\
	&\leq\sum_{j\in \mathcal J(b)}\sum_{k\in\mathcal T} \frac{k_j+1}{2^N}P^\eta(T(\eta_{X_m}(m))=k\,|\,\eta(m-1),X_1,X_m)\,.
\end{align}
On $\mathcal E_{m-1}$, using Markov property, we can apply Lemma \ref{lemma2} with $M=0$, $t=1$, $L=\lfloor \phi_T \rfloor$ if $m=1$ and $L=\lfloor \phi_T^{1/2}\rfloor$ if $m\geq 2$ to control the probability in the last display, so that if $T$ is large enough we have, since $\gamma \geq \phi_T^{3}$,
\begin{align*}
P^\eta(T(\eta_{X_m}(m))=&k\,|\,\eta(m-1),X_1,X_m)\\
	&\leq \Big(1+c\Big(\frac{L^{d+1}}{\gamma^{\frac{d+1}{2}}}+\frac{1}{\gamma^{\frac12-\alpha}}\Big)\Big)
		\sup\Big\{   \langle \eta(m-1) \rangle_{z,L'}^k:\,{L'\geq L},z\sim X_m \Big\}\\
	&\leq  \big(1+c\frac{1}{\phi_T}\big) (p_k+\epsilon_{\phi_T^{1/2}})
	= p_k+\delta_T,
\end{align*}
where the last inequality is valid for any $m\geq 1$ and where $\delta_T$ is some function that goes to $0$ as $T$ grows.
Putting this back into \eqref{eq:new} we obtain that on $\mathcal E_{m-1}$
\begin{align}
P^\eta(X_{m+1}^{\vec v}-X_m^{\vec v}\leq b\,|\,\eta(m-1),X_1,X_m)
	&\leq \sum_{j\in J(b)}
		\sum_{k\in\mathcal T} \frac{k_j+1}{2^N}p_k+c\,\delta_T   \nonumber \\
	&\leq \sum_{j\in J(b)}
		\sum_{k\in\mathcal T} \frac{k_j}{2^N}p_k+\frac{| \mathcal J(b)|}{2^N}+c\,\delta_T, \label{eq:reprise}
\end{align}		
where $| \mathcal J(b)|$ is the cardinality of $ \mathcal J(b)$.
We notice that $\forall j\not=j_{\min}$ we have $ \sum_{k\in\mathcal T}\frac{k_j}{2^N}p_k\leq P^\eta(Y_m=v_j)+\delta$. On the other hand, $ \sum_{k\in\mathcal T}\frac{k_{j_{\min}}}{2^N}p_k\leq P(Y_m=v_{j_{\min}})-(2d-1)\delta$. Hence on $\mathcal E_{m-1}$,
\begin{align*}
P^\eta(X_{m+1}^{\vec v}-X_m^{\vec v}\leq b\,|\,\eta(m-1),X_1,X_m)
	&\leq P(Y_m\leq b)-(2d-|\mathcal J(b)|)\delta+\frac{|\mathcal J(b)|}{2^{N}}+c\,\delta_T\\
	&\leq P(Y_m\leq b)
\end{align*}
if $T$ is big enough and $|\mathcal J(b)|\leq 2d-1$, since we chose $\delta>2d/2^N$. This last inequality is trivial in the case $|J(b)|=2d$, that is $b\geq v_{\max}$, as both terms are equal to $1$. Plugging the last display back into \eqref{eq:but}, we easily obtain \eqref{eq:m_goodorbad}.

%
\medskip
We prove the ``other half'' of the lemma in a completely specular way. We need 
\begin{align}\label{eq:other_half_lemma}
P^\eta(X_T^{\vec v}\geq (v+\varepsilon)T)
	\leq \tfrac 12\,{\rm e}^{-\phi_T^{1/4}},
\end{align}
since this combined with \eqref{eq:half_lemma} gives the final result. As before, we bound the first step with the worst case 
$$
P^\eta(X_T^{\vec v}\geq (v+\varepsilon)T)
	\leq P^\eta(X_T^{\vec v}-X_1^{\vec v}\geq (v+\varepsilon)T-1)
$$
and define a new sequence of i.i.d.~random variables $(\tilde Y_k)_{k\in\N}$ with values in $\{v_{-d},...,v_{d}\}$, independent of the interchange process and of the walker and such that
\begin{align*}
P^\eta(\tilde Y_1=v_j)=
\begin{cases}
\E_\mu [s_j]-\delta \qquad & \mbox{for }j\not=j_{\max} \\
\E_\mu [s_{j_{\max}}]+(2d-1)\delta\qquad & \mbox{for }j=j_{\max}.
\end{cases}
\end{align*}
Similar comments (and eventually corrections) made for the $Y_j$'s hold for the $\tilde Y_j$'s, that is, the $\tilde Y_j$'s are (or can be easily made) well defined and $E^\eta[\tilde Y_1]\in(v,v+\varepsilon/2)$. We keep the previous definition of the good events $(\mathcal E_m)_{m\in\N}$ and aim to prove
\begin{align}\label{eq:other_bye_badevents}
P^\eta(X_T^{\vec v}-X_1^{\vec v}\geq (v+\varepsilon)T-1) 
	&\leq P^\eta\Big(\sum_{m=1}^{T-1}\tilde Y_m\geq (v+\varepsilon)T-1\Big) +\sum_{m=1}^{T-1}\P^\eta(\mathcal E^c_{m-1}).
\end{align}
As we have seen before, we can use a mild large deviations argument as in \eqref{eq:ld_iid} combined with \eqref{eq:unfavorable} to bound \eqref{eq:other_bye_badevents} and conclude \eqref{eq:other_half_lemma}.

\medskip

The proof of \eqref{eq:other_bye_badevents} is similar to the one of \eqref{eq:bye_badevents}. We proceed in the same way as before and obtain the equivalent of \eqref{eq:reprise} that now reads as follows:
For any $b\in \R$, $m\geq 1$, on the event $\mathcal E_{m-1}$ we have that
\begin{align*}
P^\eta(X_{m+1}^{\vec v}-X_m^{\vec v}\geq b\ |\ \eta(m-1),X_1,X_m)&
	\leq \sum_{j\notin \mathcal J(b)}
		\sum_{k\in\mathcal T} \frac{k_j}{2^N}p_k+\frac{2d-|\mathcal J(b)|}{2^N}+c\,\delta_T\ ,
\end{align*}		
where again $\delta_T$ is some function that goes to $0$ when $T$ grows to infinity.
%
%
%
For all $ j\not = j_{\max}$ we have $ \sum_{k\in\mathcal T}\frac{k_j}{2^N}p_k\leq P(Y_m=v_j)+\delta$, while $ \sum_{k\in\mathcal T}\frac{k_{j_{\max}}}{2^N}p_k\leq P(Y_m=v_{j_{\max}})-(2d-1)\delta$. Hence,
\begin{align*}
P^\eta(X_{m+1}^{\vec v}-&X_m^{\vec v}\geq b\,|\,\eta(m-1),X_1,X_m)\\
	&\leq P(Y_m\geq b)-|\mathcal J(b)|\delta+(2d-|\mathcal J(b)|)2^{-N}+c\,\delta_T\,\leq \,P(Y_m\geq b)
\end{align*}
if $T$ is big enough and $|\mathcal J(b)|\geq 1$ as $\delta>2d/2^{N}$ (here again, if $|\mathcal J(b)|=0$, that is $b<v_{\min}$ the inequality is trivial). From there we prove easily the equivalent of \eqref{eq:but} with the opposite inequality and conclude the proof of \eqref{eq:other_bye_badevents}.
%
%
%
%
\end{proof}

\subsection{Renormalization step}
\begin{proposition}
\label{prop:renorm}
Take any $\varepsilon>0$. Let $N$, $T$ and $\gamma$ be large enough.  Given $\eta\in\mathcal S^{\Z^d}$ and $t\geq T$ such that
$$
y\in G(\eta,\phi_t)\qquad \forall y\in B_t, 
$$
it holds that
\begin{align*}
P^\eta\big(|{X_t^{\vec v}}-vt|\geq\varepsilon t\big )
	\leq {\rm e}^{-\phi_t^{1/4}}.
\end{align*}
\end{proposition}
\begin{proof}
We consider $T$ so large that the conclusion of Proposition \ref{lemma4} holds for the time $T^{1/3}$ and consider $t\geq T$.
We first prove that 
\begin{align}\label{cane}
P^\eta\big(X_t^{\vec v}<(v-\varepsilon)\, t\big)\leq \tfrac 12\,{{\rm e}^{-\phi_t^{1/4}}}.
\end{align}
We define a sequence of times $(t_n)_{n\in\N}$ such that $t_0\in[T^{1/3},T]$ and, for $n\geq 0$,
\begin{align}\label{eq:time}
t_{n+1}\in[t_n^2,(t_n+1)^2]\,,
\end{align}
so that for some $\tilde n\geq 0$ we have $t_{\tilde n}=t$ (for example define $t_{\tilde n}=t$ for a suitable $\tilde n$ and then define recursively $t_{k-1}=\lfloor {\sqrt t_k}\rfloor$ until reaching the interval $[T^{1/3},T]$). We also define a sequence of positive real numbers $(c_n)_{n\in\N}$ such that
\begin{align*}
v\geq c_n:=\frac 12\Big(3-\sum_{k=0}^n\big(\frac{\varepsilon}{1+\varepsilon}\big)^k\Big)\,v\geq (1-\varepsilon/2) \,v\qquad \mbox {for }n\in\N.
\end{align*}
Note that $(c_n)_{n\in\N}$ is a decreasing sequence and that $c_0=v$ and $\lim_{n\to\infty}c_n=v(1-\varepsilon/2)$. We want to prove by induction that, given $n\in\N$ and $\eta\in\mathcal S^{\Z^d}$, 
\begin{align}\label{passo_induttivo}
\mbox{\textit{if $\forall y\in B_{t_n}$ it holds that $y\in G(\eta, \phi_{t_n})$, then}}\,
P^\eta(X^{\vec v}_{t_n}< c_n t_n)\leq \tfrac 12\, {\rm e}^{-\phi_{t_n}^{{1}/4}}.
\end{align}
Proposition \ref{lemma4} takes care of the initialization step $n=1$. We assume now that \eqref{passo_induttivo} holds for some $n\geq 1$ and we show that this implies the case $n+1$. 
By \eqref{eq:time} we can write $t_{n+1}=t_n(t_n-1)+r$, where the rest $r$ is a number between, say, $t_n$ and $4t_n$. We first have to wait an initial time lapse, since we can not use the iteration assumption at time 0 as the maximal trap could be of order $\phi_{t_{n+1}}$ instead of  $\phi_{t_{n}}$:
\begin{align*}
P^\eta(X_{t_{n+1}}^{\vec v}< c_{n+1}t_{n+1})\leq P^\eta( X_{t_{n+1}}^{\vec v}-X_r^{\vec v}< c_{n+1}t_{n+1}+r)\,.
\end{align*}
We define on the same probability space (enlarged if necessary) a sequence 
$(Z_k)_{k\in\N}$ of i.i.d.~random variables independent from $\eta$, from the interchange process and from the walker, with distribution
\begin{align*}
P^\eta(Z_k=c_nt_n)=1-{\rm e}^{-\phi_{t_n}^{1/4}},\qquad P^\eta(Z_k=-t_n)={\rm e}^{-\phi_{t_n}^{1/4}}.
\end{align*}
Using our inductive hypothesis we aim to show that
\begin{align}\label{galoppo}
P^\eta(& X_{t_{n+1}}^{\vec v}-X_r^{\vec v}< c_{n+1}t_{n+1}+r)\nonumber\\
	&\leq P^\eta\Big(\sum_{k=1}^{t_n-1}Z_k\leq c_{n+1}t_{n+1}+r\Big)
		+\sum_{\substack{y\in B_{t_{n+1}}\\r\leq s \leq t_{n+1}}}
		\P^\eta\big(y\not\in G(\eta(s),\phi_{t_n})\big)\,.
\end{align}
To obtain this, 
we just show that, for all $m\geq 1$ and for any $a\in\R$, we have
\begin{align}\label{trotto}
P^\eta(& X_{mt_n+r}^{\vec v}-X_r^{\vec v}\leq a)
	\leq  P^\eta( X_{(m-1)t_{n}+r}^{\vec v}-X_r^{\vec v}+Z_m\leq a)
		+\P^\eta(\mathcal E^c_{m-1})\,,
\end{align}
where for $m\geq 0$
\begin{align*}
\mathcal E_{m}:=\{y\in G(\eta(mt_n+r),\phi_{t_n}),\,\forall y\in B_{(m+1)t_n+r}\}
\end{align*}
is the event of having a favorable environment at step $m$ (i.e., at time $mt_n+r$).

\smallskip

Let us prove \eqref{trotto}:
\begin{align*}
&\qquad P^\eta( X_{mt_n+r}^{\vec v}-X_r^{\vec v}\leq a)
	= P^\eta( X_{mt_n+r}^{\vec v}-X_{(m-1)t_n+r}^{\vec v}+X_{(m-1)t_n+r}^{\vec v}-X_r^{\vec v}\leq a)\\
	&\leq \sum_{x,y\in\Z^d}P^\eta( X_{mt_n+r}^{\vec v}-X_{(m-1)t_n+r}^{\vec v}\leq a-\langle y-x,\vec v\rangle\,|\,A_{x,y})P^\eta(X_{(m-1)t_n+r}=y,X_r=x)
		+\P^\eta(\mathcal E^c_{m-1})
\end{align*}
with 
$
A_{x,y}:=\{X_{(m-1)t_n+r}=y,X_r=x,\mathcal E_{m-1}\}.
$
Thanks to the inductive step we have that
\begin{align*}
P^\eta( X_{mt_n+r}^{\vec v}-X_{(m-1)t_n+r}^{\vec v}\leq a-\langle y-x,\vec v\rangle\,|\,A_{x,y})
	\leq P^\eta(Z_m\leq a-\langle y-x,\vec v\rangle),
\end{align*}
and therefore 
\begin{align*}
P^\eta(& X_{mt_n+r}^{\vec v}-X_r^{\vec v}\leq a)
	\leq P^\eta(Z_m+ X_{(m-1)t_n+r}^{\vec v}-X_r^{\vec v}\leq a)+\P^\eta(\mathcal E^c_{m-1})\,.
\end{align*}
This proves \eqref{trotto}. Iterating \eqref{trotto} and bounding $\sum_{j=1}^{t_n-1}\P^\eta(\mathcal E^c_j)$ with the union bound of the probability that no trap bigget than $\phi_{t_n}$ appears after time $r$ and until time $t_{n+1}$ in the whole ball of radius $t_{n+1}$ gives \eqref{galoppo}. 

\smallskip

We are left to bound the r.h.s.~of \eqref{galoppo}.  The first summand is equal to
\begin{align}\label{polio}
 P^\eta\Big(\sum_{k=1}^{t_n-1}\big(Z_k&-E^\eta[Z_k]\big)\leq c_{n+1}t_{n+1}+r-E^\eta[Z_k](t_n-1)\Big)\nonumber\\
	&\leq P^\eta\Big(\sum_{k=1}^{t_n-1}\big(Z_k-E^\eta[Z_k]\big)\leq (t_{n}-1)\big(c_{n+1}t_n- E^\eta[Z_k]\big)+8t_n\Big)\nonumber\\
	&\leq P^\eta\Big(\sum_{k=1}^{t_n-1}\big(Z_k-E^\eta[Z_k]\big)\leq -\,t_{n}^{3/4}(t_n-1)\Big)
		\leq{\rm e}^{-c\,t_n^{1/2}}\leq {\rm e}^{-\phi_{t_n}^{1/3}}\,,
\end{align}
where we have used the fact that $c_{n+1}<1$ and that $r\leq 4t_n$ for the first inequality. The second inequality just comes from the definitions of $t_n$ and $c_n$. We point out that for this inequality there might be the need to choose a $T$ bigger than the one considered until this point, depending on $\varepsilon$; this constitutes no problem, since in this case it will be sufficient to repeat the proof with a bigger $T$ (and hence a bigger $\gamma$). The last inequality in \eqref{polio} is just a classical concentration inequality for sums of independent bounded random variables (note that $|Z_k|\leq t_n$ a.s.).

%
%

\smallskip

For the second summand in the r.h.s.~of \eqref{galoppo} we assume (see \eqref{passo_induttivo}) that $y\in G(\eta, \phi_{t_{n+1}})$ for any $y\in B_{t_{n+1}}$ and notice that, by the choice made in \eqref{cento}, $\gamma r\geq \gamma t_n\geq\phi^3_{t_{n+1}}$. We can therefore  use Proposition \ref{proposition1} with $L=\phi_{t_{n+1}}$ and $J=\phi_{t_n}$ (we can suppose to have taken $T$ large enough such that the condition ``$L(\epsilon_J-\epsilon_L)$ large enough" is satisfied) to bound
\begin{align}\label{cheo}
\sum_{\substack{y\in B_{t_{n+1}}\\r\leq s \leq t_{n+1}}}
		P^\eta\big(y\not\in G(\eta(s),\phi_{t_n})\big)
	\leq c_1 t_{n+1}^{d+1}\frac{{\exp}\{-c_2\,\phi_{t_n}^d(\epsilon_{\phi_{t_n}}-\epsilon_{\phi_{t_{n+1}}})^2\}}{(\epsilon_{\phi_{t_n}}-\epsilon_{\phi_{t_{n+1}}})^2}\leq {\rm e}^{-\phi_{t_n}^{d/3}},
\end{align}
where the last bound comes from the explicit expression \eqref{epsilon}. Plugging \eqref{cheo} and \eqref{polio} into \eqref{galoppo} gives the conclusion of \eqref{passo_induttivo}, so that \eqref{cane} is proven. 

\medskip

For the converse
\begin{align*}
P^\eta\big(X_t^{\vec v}>(v+\varepsilon) t\big)
	\leq \tfrac 12\,{{\rm e}^{-\phi_t^{1/4}}}
\end{align*}
we follow a very similar strategy. We take the same definition of $t_n$ and a new sequence of $c_n$'s, this time increasing:
\begin{align*}
c_n:=\frac{1}{2}\Big(1+\sum_{k=0}^n \big( \frac{\varepsilon}{1+\varepsilon}\big)^k\Big)\,v\,.
\end{align*}
The inductive step in this case will be
\begin{align*}
\mbox{\textit{if $\forall y\in B_{t_n}$ it holds that $y\in G(\eta, \phi_{t_n})$, then}}\,
P^\eta(X^{\vec v}_{t_n}> c_n t_n)\leq \tfrac12\, {\rm e}^{-\phi_{t_n}^{{1}/4}}.
\end{align*}
We define new random variables $(Z_k)$ such that $Z_k=c_nt_n$ with probability $1-{\rm e}^{-\phi_{t_n}^{1/4}}$ and $Z_k=t_n$ with probability ${\rm e}^{-\phi_{t_n}^{1/4}}$. The events $\mathcal E_{m-1}$ stay the same. We can now just follow the same steps as before. The only slight difference is with estimate  \eqref{polio}, which now reads
\begin{align*}
  P^\eta\Big(\sum_{k=1}^{t_n-1}Z_k\geq c_{n+1}t_{n+1}+r\Big)
	&\leq P^\eta\Big(\sum_{k=1}^{t_n-1}\big(Z_k-E^\eta[Z_k]\big)\geq (t_{n}-1)\big(c_{n+1}t_n- E^\eta[Z_k]\big)\Big)\nonumber\\
	&\leq P^\eta\Big(\sum_{k=1}^{t_n-1}\big(Z_k-E^\eta[Z_k]\big)\geq t_{n}^{3/4}(t_n-1)\Big)
		\leq{\rm e}^{-c\,t_n^{1/2}}\,,
\end{align*}
which holds with the same arguments as before.
\end{proof}

\subsection{Proof of Proposition \ref{prop:projection}} \label{proiezione}
%

We fix $\varepsilon >0$ and pick $N$, $T$ and $\gamma$ large enough so that Proposition \ref{prop:renorm} holds. 
Since under $\P$ the initial distribution for the environment is a product measure, by a union bound we have that 
\begin{align}
\P(\exists y \in B_t \textrm{ s.t.~} y\notin G(\eta,\phi_t)) 
	&\leq \sum_{y\in \mathrm B_t} \sum_{L \ge \phi_t} \sum_{k\in\mathcal T}  \P 			\big( |\langle \eta(t) \rangle^k_{y,L} -p_k|> \epsilon_{\phi_t}\big)\nonumber\\
	&\leq\sum_{y\in \mathrm B_t} \sum_{L \ge \phi_t} \sum_{k\in\mathcal T:\,p_k>0} {\rm e}^{- c L^d\epsilon^2_{\phi_t}}
	\leq \,|\mathcal T|\, t^d \, \frac{{\rm e}^{-c\,  \phi_t \epsilon^2_{\phi_t}}}{c\,\epsilon^2_{\phi_t}} 
	\leq c_1\, {\rm e}^{-t^{c_2}}\,.\nonumber
\end{align}
By the Borel-Cantelli lemma, it then follows from Proposition \ref{prop:renorm} that $P-a.s.$, eventually ${X_t^{\vec v}}/{t}\in(v-\varepsilon,v+\varepsilon)$.


\smallskip

\appendix

\section{}\label{appendix}


\subsection{Proof of Claim  \ref{claim:heatkernel_decrease}}\label{appendix1}
Take $f:\Z^d\to\R$ such that $f(x)-f(x+e_i)\geq 0$ for all $x\in\Z^d$ with $\langle x,e_i\rangle\geq 0$ and such that $f(x)=f(\tilde x_i)$ for all $x\in\Z^d$, where $\tilde x_i:=(x_{1},x_{2},...,-x_i,...,x_d)$. Consider $f(t,x)$, the unique solution of
\begin{align*}
\begin{cases} 
	\partial_tf(t,x)&=\gamma\Delta f(t,x) \quad\forall t\geq 0 \\ f(0,x)&=f(x)\,.
\end{cases}
\end{align*}
Then we have that (1) $f(t,x)=f(t,\tilde  x_i)$ for all $ t\geq 0,\,x\in\Z^d$ and (2) $f(t,x)-f(t,x+e_i)\geq 0$ for all $ t\geq 0$ and $ x\in\Z^d$ with $\langle x,e_i\rangle>0$. The first property is true by symmetry. The second can be explained by noticing, for example, that the continuous-in-time function
\begin{align*}
g(t,x):=\begin{cases} 
	f(t,x)-f(t,x+e_i) \quad\mbox{if}\, \langle x,e_i\rangle\geq0\\ 
	f(t,x)-f(t,x-e_i)\quad\mbox{otherwise}
\end{cases}
\end{align*}
is clearly everywhere positive for $t=0$, but also stays positive for all $t>0$. In fact, take a time $t$ such that $g$ is positive everywhere; then, in the points $x$ where $g(t,x)=0$, one has $\partial_tg(t,x)\geq 0$, as it is easy to verify.
To conclude the proof, just take $f(t,x)=p_M(t,x)$. 

\subsection{Proof of the Corollary \ref{cor:crowns}}\label{appendix2}
If $y$ and $y'$ are in the same crown there is nothing to prove. Suppose then that $y'\in C_{n'}$ and $y\in C_n$ with $n'>n$. Consider $\hat x_{n'}$ and an $e\in\mathcal N$ such that $\langle \hat x_{n'},e\rangle>0$ and $\hat x_{n'}-me$  belongs to $C_n$ for some $m\in\N$. Applying Claim \ref{claim:heatkernel_decrease} $m$-times, we have that $p_M(t,\hat x_{n'}-me)\geq p_M(t,\hat x_{n'})$ and therefore $p^+_M(t,y)\geq p^+_M(t,y')$. The same argument reversed gives the $p^-_M(\cdot,\cdot)$ case.
%

\subsection{Proof of Claim \ref{climatizzo}}\label{appendix3}

We focus on \eqref{dumbo} and split the sum into three parts: 
\begin{align}\label{acca}
\sum_{n\geq L}E_M(t,n)|C_n|
	&=\sum_{n= L}^{M_--1}E_M(t,n)|C_n|+\sum_{n=M_-}^{M_+-1}E_M(t,n)|C_n|+\sum_{n= M_+}^{\infty}E_M(t,n)|C_n|\nonumber\\
	&=:H_1+H_2+H_3\,,
\end{align}
where we have called $M_-:=\big\lfloor M-(\gamma t)^{1/2+\varepsilon}\big\rfloor\vee L$ and $M_+:=\big\lfloor M+(\gamma t)^{ 1/2+\varepsilon}\big\rfloor\vee L$ (note that the first two sums might be empty). 
Part $H_3$ falls into a large deviations regime, since it can be dominated by the probability that a walker starting from far away enters the ball of radius $M$ by time $t$:
\begin{align*}
H_3=\sum_{n=M_+}^\infty E_M(t,n)|C_n|
	&\leq c\sum_{n=M_+}^\infty p_M(t,\hat x_n)\,n^{d-1}
	 = c \sum_{n=M_+}^\infty\frac{n^{d-1}}{M^d}P_{\hat x_n}(Y_t\in B_M),
\end{align*}
where  $(Y_t)_{t\geq 0}$ is a simple random walk that jumps at rate $\gamma t$. 
Notice that $P_{\hat x_n}(Y_t\in B_M)\leq P_{0}(|Y_t|\geq n-M)\leq {\rm e}^{-c_1\, (n-M)^2/(\gamma t)}$, where for the last inequality we have applied the continuous-time version of the large deviations estimate \cite[Eq.~2.7]{LL10}. Hence, as can be easily checked,
\begin{align}\label{eq:small_n}
H_3	&\leq c \sum_{n=M_+}^\infty\frac{n^{d-1}}{M^d}{\rm e}^{-c_1 (n-M)^2/(\gamma t)}
	\leq c_2\,{\rm e}^{-c_3 (\gamma t)^{2\varepsilon}}.
\end{align}
Part $H_1$, if not empty, 
can be treated similarly as it can be dominated by the probability that a walker that starts well inside the ball of radius $M$ will leave the ball by time $t$:
\begin{align}\label{eq:large_n}
H_1=\sum_{n=L}^{M_-} E_M(t,n)|C_n|
	& \leq  c \sum_{n=L}^{M_-}\frac{n^{d-1}}{M^d}\big(1-P_{\check x_n}(Y_t\in B_M)\big)
	\leq  c_2\,{\rm e}^{-c_3 (\gamma t)^{2\varepsilon}}.
\end{align}

\smallskip

Part $H_2$, when not empty, is more delicate.
For $M_-\leq n < M_+$ we write
\begin{align}\label{eq:triangular} 
E_M(n,t)&\leq |p_M(t,\hat x_n)-G_M(t,\hat x_n)|
+|G_M(t,\hat x_n)-G_M(t,\check x_n)|+|G_M(t,\check x_n)- p_M(t,\check x_n)|\nonumber\\
	&=:A_1(n)+A_2(n)+A_3(n)=A_1+A_2+A_3,
\end{align}
where  $G_M(t,x):=\frac{1}{|B_M|}\sum_{y\in B_M} G(t,x-y)$ and, for $z\in\R^d$, 
$
G(t,z):=\frac{{\rm exp}\{-{\|z\|^2}/{(2\gamma t)}\}}{{(2\pi \gamma t)}^{d/2}}
$
is the Gaussian kernel. 

\smallskip

We treat together  $A_1$ and $A_3$ from \eqref{eq:triangular}, since the calculation here below is valid for both of them (as can be seen simply by substituting $\hat x_n$ by $\check x_n$). We consider just $A_1$ for simplicity and write 
\begin{align}\label{eq:totti2}
A_1  
	\leq \frac{1}{|B_M|}\sum_{y\in \hat S_1\cup \hat S_2}{\big| p(t,\hat x_n-y)-G(t,\hat x_n-y)\big|}\,,
\end{align}
where $\hat S_1=\hat S_1(n):=\{y\in B_M:\, \|y-\hat x_n\|\leq (\gamma t)^{\frac 12+\varepsilon}\}$ and $\hat S_2=\hat S_2(n):=\{y\in B_M:\, \|y-\hat x_n\|> (\gamma t)^{\frac 12+\varepsilon}\}$.
The local central limit theorem (see, e.g., \cite[Eq.~(2.5)]{LL10}) says that $|p(t,z)-G(t,z)|\leq  c\,{(\gamma t)^{-d/2-1}}$ for all $z\in\Z^d$. Therefore
\begin{align*}
\sum_{y\in \hat S_1}{\big| p(t,\hat x_n-y)-G(t,\hat x_n-y)\big|}
	&\leq \frac c{(\gamma t)^{d/2+1}}\Big({(\gamma t)^{(1/2+\varepsilon)d}}\wedge |B_M|\Big)\,.
\end{align*}


\noindent The sum over $\hat S_2$ can be bounded by
\begin{align*}
\sum_{y\in \hat S_2}  p(t,\hat x_n-y)+G(t,\hat x_n-y)
	&\leq P\big(\|Y_t\|>(\gamma t)^{\frac12+\varepsilon}\big)+P\big(\|Z\|>(\gamma t)^{\frac12+\varepsilon}\big)
	\leq c_1\, {\rm e}^{-c_2 (\gamma t)^{2\varepsilon}},
\end{align*}
where $(Y_t)_{t\geq 0}$ is a simple random walk on $\Z^d$ starting at the origin that jumps at rate $\gamma t$ and $Z$ is a $d$-dimensional normal random variable with zero mean and covariance $\gamma t\cdot{\rm Id}$.
Inserting the last two displays into \eqref{eq:totti2} shows that 
\begin{align}\label{eq:wedge2}
\sum_{n=M_-}^{M_+-1}(A_1+A_3)|C_n|
	&\leq \frac{c_1}{(\gamma t)^{\frac d2+1}}\Big(\frac{(\gamma t)^{(\frac12+\varepsilon)d}}{M^d}\wedge 1\Big)\sum_{n=M_-}^{M_+}n^{d-1}
	\leq  c_2\,(\gamma t)^{-1+\varepsilon d}\,,
\end{align}
as can be checked by distinguishing the case $M>t^{1/2+\varepsilon}$ and $M\leq t^{1/2+\varepsilon}$.

\smallskip

We estimate now $A_2$. Using the definition of $\hat S_1$ and $\hat S_2$ given above and defining 
$\check S_1=\check S_1(n):=\{y\in B_M:\, \|y-\check x_n\|\leq (\gamma t)^{\frac 12+\varepsilon}\}$ and $\check S_2=\check S_2(n):=\{y\in B_M:\, \|y-\check x_n\|> (\gamma t)^{\frac 12+\varepsilon}\}$, we write
\begin{align}\label{autostrada}
A_2
	\leq \frac{1}{|B_M|}\Big(\Big|\sum_{y\in\hat S_1} &G(t,\hat x_n-y) 
		- \sum_{y\in\check S_1} G(t,\check x_n-y) \Big|\nonumber\\
		&+ \sum_{y\in \hat S_2}  G(t,\hat x_n-y)+\sum_{y\in \check S_2}G(t,\check x_n-y)\Big)\,.
\end{align}
The sums over $\hat S_2$ and $\check S_2$ can be bounded by 
\begin{align}\label{touch}
\sum_{y\in \hat S_2}  G(t,\hat x_n-y)+\sum_{y\in \hat S_2}G(t,\hat x_n-y)
	\leq 2\, P\big(\|Z\|>(\gamma t)^{1/2+\varepsilon} )\leq c_1\, {\rm e}^{-c_2 (\gamma t)^{2\varepsilon}},
\end{align}
where $Z$ is again a $d$-dimensional normal random variable with zero mean and variance $\gamma t\cdot {\rm Id}$. 
We are left with the sums over $\hat S_1$ and $\check S_1$ of \eqref{autostrada}. We use once more the triangular inequality to obtain
\begin{align}\label{chiera}
\Big|&\sum_{y\in\hat S_1} G(t,\hat x_n-y) - \sum_{y\in\check S_1} G(t,\check x_n-y)\Big|
	\leq \sum_{y\in\hat S_1} |G(t,\hat x_n-y) - \bar G(t,\hat x_n-y)|\nonumber \\
		&+ \Big|\sum_{y\in\hat S_1} \bar G(t,\hat x_n-y)- \sum_{y\in\check S_1} \bar G(t,\check x_n-y)\Big|	+ 	\sum_{y\in\check S_1} |G(t,\check x_n-y) - \bar G(t,\check x_n-y)|\,,
\end{align}
where we have called, for $z\in\R^d$, $\bar G(t,z)=\int_{K(z)}G(t,w)\,dw$, and $K(z)$ is the cube of size one centered in $z$. The first and third term of \eqref{chiera} can be roughly bound thanks to the mean value theorem: For each $y\in\hat S_1$ 
\begin{align*}
|G(t,\hat x_n-y) - \bar G(t,\hat x_n-y)|
	&\leq \Big|\int_{K(\hat x_n-y)}G(t,\hat x_n-y)-G(t,w)\,dw\,\Big|\\
	&\leq c_1\frac{\|\zeta\|}{(\gamma t)^{d/2+1}}{\rm e}^{-\frac{\|\zeta\|^2}{2\gamma t}}
	\leq\frac{c_2}{(\gamma t)^{(d+1)/2}}\,,
\end{align*}
with $\zeta$ some point in $K(\hat x_n-y)$. Hence, distinguishing between the case $M\leq (\gamma t)^{1/2+\varepsilon}$ and $M>(\gamma  t)^{1/2+\varepsilon}$,
\begin{align}\label{figliadellamore}
\sum_{y\in\hat S_1} |G(t,\hat x_n-y) - \bar G(t,\hat x_n-y)|
	\leq c\,\frac{(M\wedge (\gamma t)^{1/2+\varepsilon})^d}{(\gamma t)^{(d+1)/2}}\,,
\end{align}
and the same bound holds for terms of the sum in $\check S_1$. 

For bounding the middle term in \eqref{chiera}, we call $\hat D=\hat D(n):=\bigcup_{y\in \hat S_1}K(\hat x_n-y)\subset\R^d$ and $\check D=\check D(n):=\bigcup_{y\in \check S_1}K(\check x_n-y)\subset\R^d$ and notice that $\sum_{y\in\hat S_1} \bar G(t,\hat x_n-y)$ is exactly the probability that a Brownian Motion $(W_t)$ started at the origin ends up in $\hat D$ after time $\gamma t$ (the same for $\check D$). Call $\mathcal R$ the rotation that brings $\check x_n$ on the half-line from the origin and passing through $\hat x_n$ and let $\check D':=\mathcal R(\check D)$. 
By rotation and translation invariance of $W_t$ we have then that 
\begin{align}\label{merola}
\Big|\sum_{y\in\hat S_1} \bar G(t,\hat x_n-y)- \sum_{y\in\check S_1} \bar G(t,\check x_n-y)\Big| 
	=\big|P(W_t\in \hat D\setminus \check D')-P(W_t\in\check D'\setminus \hat D)\big|\,
\end{align}
and we further notice that  $\hat D\setminus \check D'$ and $\check D'\setminus \hat D$ have volume smaller than $c\,{(M\wedge(\gamma  t)^{1/2+\varepsilon})^{d-1}}$, i.e.~the smallest between the $d-1$-dimensional surface of $B_M$ and the $d-1$-dimensional  surface of $B_{t^{1/2+\varepsilon}}$ times a constant that depends on the dimension but not on $M$ or $n$.

In the case $M>(\gamma t)^{1/2+\varepsilon}$, expression
\eqref{merola} can therefore be bounded by $c_1\,{(\gamma  t)^{-1/2+\varepsilon(d-1)}}$. Using this estimate and \eqref{figliadellamore} for bounding \eqref{chiera}, and putting this together with \eqref{touch} back into \eqref{autostrada}, we obtain
\begin{align}\label{astro}
 \sum_{n=M_-}^{M_+-1}A_2(n)|C_n|
	\leq c_1\,\frac{(\gamma  t)^{-\frac12+\varepsilon(d-1)}}{M^d}\sum_{n=M_-}^{M_+}n^{d-1}
	\leq c_2\,\frac{(\gamma  t)^{\varepsilon d}}{M}
	\leq c_3\,{(\gamma  t)^{-1/2+\varepsilon d}}.
\end{align}

In the case $M\leq (\gamma t)^{1/2+\varepsilon}$ we further notice the following: For $n\in[L,t^{1/2+\varepsilon}]$, $\hat D$ and $\check D$ are exactly the same object,  so that  $\hat D$ and $\check D'$ have the same volume. It follows that also $\hat D\setminus \check D'$ and $\check D'\setminus \hat D$ have the same volume, and this volume is smaller than $c\,{M^{d-1}}$ as explained before. Hence, for these $n$'s, we can bound \eqref{merola} by
\begin{align*}
\Big|\int_{\hat D\setminus \check D'}G(t,z)dz - \int_{\check D'\setminus \hat D}G(t,z)dz\Big| 
	&\leq c_1\,\frac{M^{d-1}}{(\gamma t)^{d/2}}\Big|{\rm e}^{-\frac{\|z_+\|^2}{2\gamma t}}-{\rm e}^{-\frac{\|z_-\|^2}{2\gamma t}}\Big|\\
	&\leq c_2 \frac{M^{d-1}( \|z_+\|-\|z_-\|)}{(\gamma t)^{\frac d2+1}} \,\xi\, {\rm e}^{-\frac{\xi^2}{2\gamma t}}
	\leq c_3  \frac{M^d}{(\gamma t)^{\frac{d+1}2}} \,,
\end{align*}
where $z_+$ and $z_-$ are, respectively, the farthest and the closest point to the origin into $\{\hat D\setminus \check D'\}\cup \{\check D'\setminus \hat D\}$ and $\xi\in[\|z_-\|,\|z_+\|]$. For the second inequality we have used the mean value theorem, while for the third we have used the fact that the distance between $z_+$ and $z_-$ is at most a constant times $M$ and then we have took the worst-case for $\xi$ (corresponding to $\xi\simeq(\gamma t)^{1/2}$). For $n\in(t^{1/2+\varepsilon},M_+-1]$ we just bound \eqref{merola} with $cM^{d-1}(\gamma t)^{-d/2}$. We obtain
\begin{align*}
 \sum_{n=M_-}^{M_+-1}A_2(n)|C_n|
	&\leq { \frac{c_1}{(\gamma t)^{(d+1)/2}}}\sum_{n=0}^{\lfloor t^{1/2+\varepsilon}\rfloor}n^{d-1} +  \frac{c_2}{M (\gamma t)^{d/2}}\sum_{n=\lfloor t^{1/2+\varepsilon}\rfloor+1}^{M_+}n^{d-1}
	\leq c_3 (\gamma t)^{-1/2+\varepsilon d}.
\end{align*}
This, together with \eqref{astro} and \eqref{eq:wedge2}, shows that $H_2\leq c\,(\gamma t)^{-1/2+\varepsilon d}$. Putting back together this bound for $H_2$, the bound for $H_3$ \eqref{eq:small_n} and for $H_1$ \eqref{eq:large_n} into \eqref{acca} gives
\eqref{dumbo}. Our estimates for $A_1$, $A_2$ and $A_3$ and \eqref{eq:triangular} also imply \eqref{dumbo2}.

\small
\bibliography{salvisimenhaus}
\bibliographystyle{alpha}

\end{document}